\RequirePackage{fix-cm}
\documentclass[smallextended]{svjour3}       
\smartqed  
\usepackage{graphicx}

\usepackage{amssymb}

\usepackage{txfonts}
\usepackage{lmodern}
\usepackage{bm}
\usepackage{array}

\newcommand{\contract}{\mathord{\varparallelinv}}

\begin{document}

\title{The Asymptotic Diameter of Cyclohedra\thanks{Research funded by Ville de Paris through {\'E}mergences project ``Combinatoire {\`a} Paris''.}
}


\author{Lionel Pournin
}


\institute{Lionel Pournin \at
              LIPN, Universit{\'e} Paris 13, Villetaneuse, France \\
              \email{lionel.pournin@univ-paris13.fr}           
}

\date{$\mbox{ }$}

\maketitle

\begin{abstract}
It is shown here that the diameter of the $d$-dimensional cyclohedron is not greater than $\lceil 5d/2\rceil-2$. It is also shown that the $5/2$ coefficient in this upper bound is asymptotically sharp. More precisely, the $d$-dimensional cyclohedron has diameter at least $5d/2-4\sqrt{d}-4$.
\end{abstract}

\section{Introduction}
\label{Csection.-1}

Cyclohedra have first been introduced by Raoul Bott and Clifford Taubes \cite{BottTaubes1994} as combinatorial objects. They have been later constructed as polytopes by Martin Markl \cite{Markl1999} and by Rodica Simion \cite{Simion2003}. In her article, Rodica Simion describes them as type-B analogues of associahedra. It turns out that cyclohedra belong to several families of polytopes such as graph-associahedra \cite{CarrDevadoss2006,Devadoss2009}, and generalized associahedra \cite{ChapotonFominZelevinsky2002,FominZelevinsky2003}. The latter family, introduced by Sergey Fomin and Andrei Zelevinsky \cite{FominZelevinsky2003}, extends associahedra to any finite Coxeter groups: cyclohedra are associahedra of type B and C, and (classical) associahedra are associahedra of type A. The third infinite subfamily of generalized associahedra is that of the associahedra of type D. All other generalized associahedra have exceptional types, and there is a finite number of them.

The asymptotic diameter of the classical associahedra has been known for more than twenty-five years: Daniel Sleator, Robert Tarjan, and William Thurston have shown that the diameter of the $d$-dimensional associahedron is $2d-4$ when $d$ is large enough \cite{SleatorTarjanThurston1988}. They have also conjectured that this diameter is equal to $2d-4$ as soon as $d>9$. This conjecture has been settled recently by the author \cite{Pournin2014}. Even more recently, Cesar Ceballos and Vincent Pilaud have shown that the diameter of the $d$-dimensional associahedron of type D is exactly $2d-2$ for all $d$ \cite{CeballosPilaud2014}. 

The last infinite subfamily of generalized associahedra whose diameter is not known exactly is that of cyclohedra. The asymptotic diameter of these polytopes is given in this article using the same techniques as in \cite{Pournin2014}. More precisely, it is shown that the diameter $\Delta$ of the $d$-dimensional cyclohedron is not greater than $\lceil 5d/2\rceil-2$ and not less than $5d/2-4\sqrt{d}-4$. Therefore this diameter grows like $5d/2$ when $d$ is large. More precisely:
$$
\lim_{d\rightarrow\infty}\frac{\Delta}{d}=\frac{5}{2}\mbox{.}
$$

The proofs will rely on the combinatorial interpretation of cyclohedra given in \cite{Simion2003}. Informally, the vertices of the $d$-dimensional cyclohedron correspond to the centrally symmetric triangulations of a polygon with $2d+2$ vertices, and its edges to flips between these triangulations. This combinatorial interpretation will be formally described in Section \ref{Csection.0}, and used in the same section to find a general upper bound on the diameter of cyclohedra. Particular pairs of centrally symmetric triangulations are further introduced in Section \ref{Csection.1}. These pairs will provide a lower bound on the diameter of cyclohedra that is asymptotically sharp. This bound will be derived at the end of Section \ref{Csection.1} from a general inequality that will be proven in Section \ref{Csection.3}. This inequality will be obtained using the same techniques as in \cite{Pournin2014}. These techniques, originally developed in the case of arbitrary triangulations of convex polygons, are adapted to centrally symmetric triangulations in Section \ref{Csection.2}. In Section \ref{Csection.4}, a short discussion on the diameter of low-dimensional cyclohedra completes the article.

\section{An upper bound on the diameter of cyclohedra}
\label{Csection.0}

Consider a positive integer $d$ and a convex polygon $\pi$ with $2d+2$ vertices. Any set of two distinct vertices of $\pi$ will be referred to as an edge on $\pi$. The elements of an edge will be called its vertices. An edge whose convex hull does not intersect the interior of $\pi$ is a boundary edge of $\pi$. Consider a vertex $x$ of $\pi$. The vertex of $\pi$ opposite $x$ will be denoted by $\bar{x}$. In other words, $x$ and $\bar{x}$ are separated by exactly $d$ vertices along the boundary of $\pi$. An edge on $\pi$ whose two vertices are opposite vertices of $\pi$ is referred to as a diagonal of $\pi$. A triangulation of $\pi$ is a maximal set of pairwise non-crossing edges on $\pi$. Note that any triangulation of $\pi$ contains all the boundary edges of $\pi$. These edges will also be referred to as the boundary edges of the triangulation. The other edges of a triangulation will be called its interior edges.

\begin{definition}
A triangulation of $\pi$ is called centrally symmetric when it contains edge $\{\bar{x},\bar{y}\}$ as soon as it contains edge $\{x,y\}$.
\end{definition}

Observe that any centrally symmetric triangulation of $\pi$ contains exactly one diagonal of $\pi$. Consider a centrally symmetric triangulation $T$ of $\pi$ and an interior edge $\{x,x'\}$ of $T$. There is a unique quadrilateral whose four boundary edges belong to $T$ and that admit $\{x,x'\}$ as one of its diagonals. Denote by $y$ and $y'$ the two vertices of this quadrilateral distinct from $x$ and from $x'$. 
\begin{figure}
\begin{centering}
\includegraphics{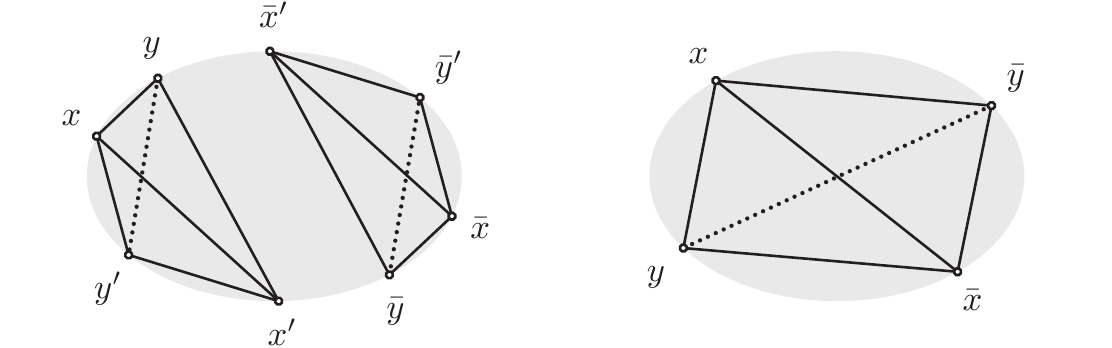}
\caption{The flip of edge $\{x,x'\}$ when $x'\neq\bar{x}$ (left), and when $x'=\bar{x}$ (right). In both cases, the edges introduced by the flip are dotted.}\label{Cfigure.flip}
\end{centering}
\end{figure}
The quadrilateral with vertices $x$, $x'$, $y$, and $y'$ is sketched in Fig. \ref{Cfigure.flip} when $x'\neq\bar{x}$ (left of the figure) and when $x'=\bar{x}$ (right of the figure). Observe that this quadrilateral is distinct from its symmetric in the former case, and coincides with it in the latter case. In particular, if $x'=\bar{x}$, then $\{x,x'\}$ and $\{y,y'\}$ are two diagonals of $\pi$. The flip operation can be defined as follows:

\begin{definition}
The operation of flipping edge $\{x,x'\}$ in triangulation $T$ consists in replacing edges $\{x,x'\}$ and $\{\bar{x},\bar{x}'\}$ within $T$ by $\{y,y'\}$ and $\{\bar{y},\bar{y}'\}$.
\end{definition}

This operation is sketched in Fig. \ref{Cfigure.flip}. Note that it results in a centrally symmetric triangulation of $\pi$ distinct from $T$. Moreover, by central symmetry, the flip of edge $\{\bar{x},\bar{x}'\}$ in $T$ is the same operation as the flip of edge $\{x,x'\}$ in this triangulation. Note that flips are defined here in order to preserve central symmetry. This condition, specific to the case of cyclohedra, is not usually required in other contexts (see for instance \cite{DeLoeraRambauSantos2010}).

Now consider the graph whose vertices are the centrally symmetric triangulations of $\pi$, and whose edges connect two triangulations when they can be obtained from one another by a flip. This graph is isomorphic to the graph of the $d$-dimensional cyclohedron (see Theorem 1 in \cite{Simion2003}). Therefore, the distance of two vertices in the graph of a cyclohedron is also the minimal number of flips one needs to perform in order to transform the centrally symmetric triangulations corresponding to these vertices into one another. In particular, one can describe the paths in the graph of cyclohedra using a succession of centrally symmetric triangulations related by flips. Consider two centrally symmetric triangulations $T^-$ and $T^+$ of a convex polygon. A path of length $k$ from $T^-$ to $T^+$ is a sequence $(T_i)_{0\leq{i}\leq{k}}$ so that $T_0=T^-$, $T_k=T^+$ and $T_i$ can be transformed into $T_{i+1}$ by a flip whenever $0\leq{i}<k$. Call $P$ the pair $\{T^-,T^+\}$. A shortest path between $T^-$ and $T^+$ will be referred to as a \emph{geodesic} between these triangulations. The distance of $T^-$ and $T^+$, or equivalently, the distance of pair $P$ is the length of any geodesic between $T^-$ and $T^+$. This distance will be denoted by $\delta(P)$ in the following of this article.

One can obtain an upper bound on the distance between two centrally symmetric triangulations of a convex polygon using the same general idea than that of the proof of Lemma 2 from \cite{SleatorTarjanThurston1988}:

\begin{theorem}\label{Ctheorem.up}
The distance of two centrally symmetric triangulations of a convex polygon with $2d+2$ vertices is not greater than $\lceil{5d/2}\rceil-2$.
\end{theorem}
\begin{proof}
Consider a convex polygon $\pi$ with $2d+2$ vertices labeled clockwise from $0$ to $2d+1$. Let $T^-$ and $T^+$ be two triangulations of this polygon. One can assume without loss of generality that the unique diagonal of $\pi$ that belongs to $T^-$ is $\{0,\bar{0}\}$. Let $x$ be the vertex of $\pi$ so that $0\leq{x}\leq{d}$ and $\{x,\bar{x}\}$ is the unique diagonal of $\pi$ contained in $T^+$. It can be assumed that $x$ is not less than $\lfloor{d/2}\rfloor+1$ by, if needed, relabeling the vertices of $\pi$ counterclockwise from $0$ to $2d+1$ in such a way that $\bar{0}$ is relabeled $0$.

First consider the centrally symmetric triangulation $U^-$ of $\pi$ sketched on the left of Fig. \ref{Cfigure.up}. The diagonal of $\pi$ that belongs to $U^-$ is $\{0,\bar{0}\}$. Above this diagonal, all the interior edges of $U^-$ are incident to vertex $0$, and below it, they are all incident to vertex $\bar{0}$. Observe that $T^-$ can be transformed into $U^-$ in at most $d-1$ flips. Indeed, if the interior edges of $T^-$ above $\{0,\bar{0}\}$ are not all incident to vertex $0$, then it is always possible to introduce a new edge incident to vertex $0$ into $T^-$ by a single flip. As $T^-$ has exactly $d-1$ interior edges above $\{0,\bar{0}\}$, it can be transformed into $U^-$ in at most $d-1$ flips.

If $x=0$, then $T^+$ can also be transformed into $U^-$ in at most $d-1$ flips using the same procedure. Hence, $T^-$ and $T^+$ have distance at most $2d-2$, and the result holds in this case. It is now assumed that $x\neq{0}$.

Consider the centrally symmetric triangulation $U^+$ of $\pi$ sketched on the right of Fig. \ref{Cfigure.up}. The diagonal of $\pi$ that belongs to this triangulation is $\{x,\bar{x}\}$. All the interior edges of $U^+$ placed on the left of this diagonal are incident to vertex $0$. By central symmetry, the edges of $U^+$ placed on the right of $\{x,\bar{x}\}$ are incident to vertex $\bar{0}$. One can transform $T^+$ into $U^+$ in at most $d-1$ flips.
\begin{figure}[b]
\begin{centering}
\includegraphics{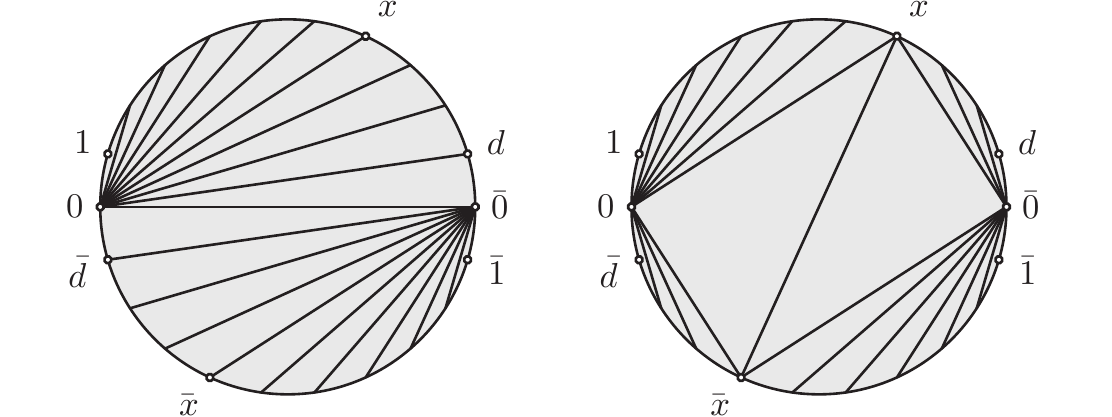}
\caption{The triangulations $U^-$ (left) and $U^+$ (right) used in the proof of Theorem \ref{Ctheorem.up}.}\label{Cfigure.up}
\end{centering}
\end{figure}
Indeed, if the interior edges of $T^+$ on the left of $\{x,\bar{x}\}$ are not all incident to vertex $0$, then it is always possible to introduce a new edge incident to vertex $0$ into $T^+$ by a single flip. As $T^+$ has exactly $d-1$ interior edges on the left of $\{x,\bar{x}\}$, it can be transformed into $U^+$ in at most $d-1$ flips.

Now observe that $U^+$ can be transformed into $U^-$ by first flipping $\{x,\bar{x}\}$, and then edges $\{\bar{0},y\}$, where $y$ ranges from $x$ to $d-1$. The first of these flips introduces diagonal $\{0,\bar{0}\}$ and the subsequent flips introduce edges $\{0,x+1\}$ to $\{0,d\}$. Therefore, $U^+$ and $U^-$ have distance at most $d-x+1$.

The above construction provides a path of length $3d-x-1$ between $T^-$ and $T^+$. As $x$ is at least $\lfloor{d/2}\rfloor+1$, the desired result holds. \qed
\end{proof}

Using the correspondence between the graphs of cyclohedra and the centrally symmetric triangulations of convex polygons and their flips, one derives an upper bound on the diameter of cyclohedra from Theorem \ref{Ctheorem.up}:

\begin{corollary}
For any positive integer $d$, the diameter of the $d$-dimensional cyclohedron is at most $\lceil{5d/2}\rceil-2$.
\end{corollary}

The remainder of the article is dedicated to finding lower bounds on the diameter of cyclohedra. In order to do this, particular pairs of centrally symmetric triangulations will be shown to have large distances. The family of these pairs is defined in the next section.

\section{Distant pairs of centrally symmetric triangulations}
\label{Csection.1}

Throughout this section, $d$ is a positive integer and $\pi$ is a convex polygon with $2d+2$ vertices, labeled clockwise from $0$ to $2d+1$. A vertex of $\pi$ will be referred to using any integer congruent to its label modulo $2d+2$. This allows using arithmetic operations on the vertices of $\pi$.

Consider the centrally symmetric triangulations $A^-$ and $A^+$ of $\pi$ respectively sketched on the left and on the right of Fig. \ref{Cfigure.0}. In this representation, only a few interior edges of $A^-$ and $A^+$ are shown and these triangulations need to be further described. Triangulation $A^-$ has a comb with $d-c+2$ teeth at vertex $0$, i.e. exactly $d-c+2$ of its interior edges are incident to vertex $0$. The teeth of the comb at vertex $0$ in $A^-$ are edge $\{0,\bar{0}\}$, and edges $\{0,x\}$ where $x$ ranges from $c$ to $d$. It is assumed in the following that a comb always has at least two teeth.  Triangulation $A^-$ has another comb at vertex $1$, whose number of teeth is $c-b+1$. Note that the interior teeth of each comb are represented as thin lines in the figure. The number of these interior teeth depends on $b$, $c$, and $d$ and is equal to zero when a comb has only two teeth. Above edge $\{1,b\}$, $A^-$ is decomposed into two portions. Within the first of these portions, the interior edges of $A^-$ form a \emph{zigzag}, i.e a simple path that alternates between left and right turns. This zigzag starts at vertex $b$ and its edges cross $\{k,\bar{k}\}$, where $k=\lfloor{b/2}\rfloor+1$. This defines the zigzag univocally. Note that vertices $2$ to $k-2$ are then incident to exactly two interior edges of $A^-$. As shown on the figure, the zigzag ends either at vertex $k-1$ or at vertex $k+1$ (depending in the parity of $b$). The portion of $A^-$ above the zigzag is made up of a single triangle with vertices $k-1$, $k$, and $k+1$. This triangle is referred to as an \emph{ear of $A^-$ at vertex $k$} because two of its edges are the boundary edges of $\pi$ incident to vertex $k$. The other edges of triangulation $A^-$ are obtained by central symmetry.

This description of $A^-$ implicitly assumes that the following inequalities hold, in order for the combs at vertices $0$ and $1$ to have at least two teeth:
\begin{equation}\label{Cequation.0.1}
b<c\leq{d}
\end{equation}

Now consider triangulation $A^+$, shown on the right of Fig. \ref{Cfigure.0}.
\begin{figure}
\begin{centering}
\includegraphics{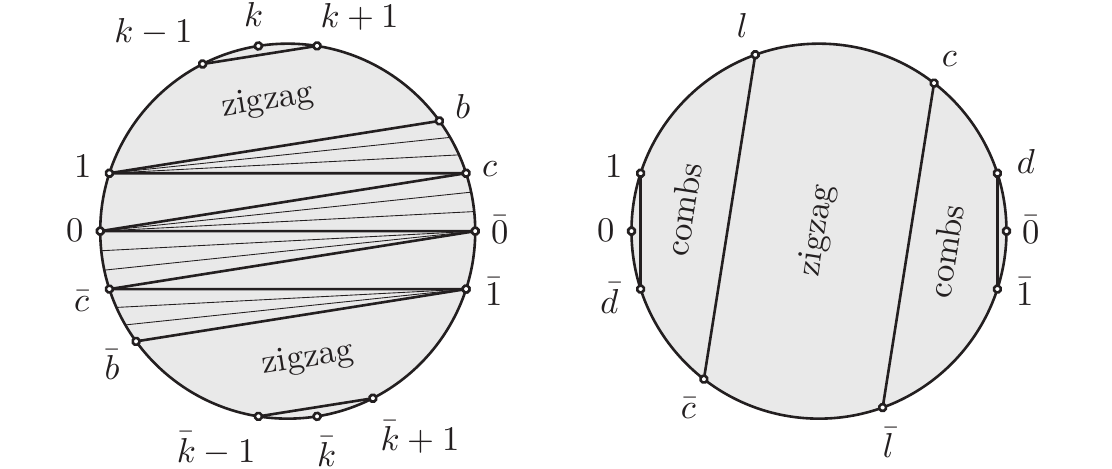}
\caption{Sketch of an $(a,b,c,d)$-pair of centrally triangulations of $\pi$, where $k$ and $l$ are respectively equal to $\lfloor{b/2}\rfloor+1$ and $a+b-c+4$. For the sake of clarity, the placement of the vertices slightly differs in the two triangulations.}\label{Cfigure.0}
\end{centering}
\end{figure}
This triangulation is decomposed into five portions. The portion left of edge $\{1,\bar{d}\}$ is an ear at vertex $0$. The interior edges of $A^+$ in the portion next to this ear are arranged into combs attached at vertices $\bar{c}$ to $\bar{d}$. Recall that each of these combs has at least two teeth. The only teeth depicted in the figure are edges $\{1,\bar{d}\}$ and $\{l,\bar{c}\}$. The other teeth are not shown because the number of combs depends on $c$ and $d$ (it is equal to $d-c+1$), and because the distribution of the teeth among these combs can be made in different ways. Within the central portion, the interior edges of $A^+$ form a centrally symmetric zigzag. This zigzag starts at the vertex labeled $l$ and its edges cross $\{0,\bar{0}\}$, which defines the zigzag univocally. Note that vertices $l$ to $c-1$ are then incident to exactly two interior edges of triangulation $A^+$. The edges of $A^+$ on the right of $\{c,\bar{l}\}$ are obtained by central symmetry.

Any interior tooth of a comb in either $A^-$ or $A^+$ will also be referred to as an interior tooth of the triangulation. Call $\tau^-$ and $\tau^+$ the number of interior teeth of respectively $A^-$ and $A^+$. These numbers are even because $A^-$ and $A^+$ are centrally symmetric. Denote:
$$
a=(\tau^-+\tau^+)/2+1\mbox{.}
$$

In other words, $a$ is obtained by summing the number of interior teeth of the combs at vertices $0$ and $1$ in $A^-$ (i.e. $d-b-1$) with the number of interior teeth of the combs at vertices $c$ to $d$ in $A^+$ (i.e. $l-2$), and by adding $1$ to the resulting quantity. In particular, one can express $l$ as follows:
$$
l=a+b-d+2\mbox{.}
$$

Observe that, by definition, $a$ must be greater by at least one than the number of interior teeth of $A^-$. Hence the following inequality holds:
\begin{equation}\label{Cequation.0.1.5}
d\leq{a+b}\mbox{.}
\end{equation}

While the quadruple $(a, b, c, d)$ completely characterizes $A^-$, several different triangulations $A^+$ may be possible for such a quadruple, depending on how the teeth are distributed among the combs in this triangulation. 

\begin{definition}
The pair $\{A^-,A^+\}$ will be called an $(a,b,c,d)$-pair when, in addition to (\ref{Cequation.0.1}) and (\ref{Cequation.0.1.5}), the following inequality holds:
\begin{equation}\label{Cequation.0.2}
a+\frac{b}{2}+1<d\mbox{.}
\end{equation}
\end{definition}

Note that inequality (\ref{Cequation.0.2}) is equivalent to $l<k$ (this can be checked using the expressions of $k$ and $l$ as functions of $a$, $b$, and $c$). Also observe that there exists an $(a,b,c,d)$-pair if and only if $(a,b,c,d)$ is a quadruple of integers satisfying (\ref{Cequation.0.1}), (\ref{Cequation.0.1.5}), and (\ref{Cequation.0.2}). In particular, there are no $(a,b,c,d)$-pairs when $d$ is less than $4$. Indeed, combining inequalities (\ref{Cequation.0.1.5}) and (\ref{Cequation.0.2}) yields $b\geq3$, and it then follows from (\ref{Cequation.0.1}) that $d$ is not less than $4$.

The following lower bound on the distance of triangulations $A^-$ and $A^+$ will be obtained at the end of Section \ref{Csection.3}:
\begin{theorem}\label{Ctheorem.1}
For any $(a,b,c,d)$-pair $A$,
$$
\delta(A)\geq3d-\left(\frac{b}{2}+\frac{2c-b}{a}+3a+5\right)\mbox{.}
$$
\end{theorem}

Most of the remainder of the article is dedicated to proving this theorem. Before doing so, it is first explained how the announced lower bound on the diameter of cyclohedra can be derived from Theorem \ref{Ctheorem.1}.

\begin{theorem}\label{Ctheorem.1.2}
If $d$ is greater than $5$, then there exists a pair of centrally symmetric triangulations of $\pi$ with distance at least $5d/2-4\sqrt{d}-4$.
\end{theorem}
\begin{proof}

Assume that $d>5$. Let $a$ be an integer so that
\begin{equation}\label{Ctheorem.1.2.eq.1}
1\leq{a}<\frac{d-1}{2}\mbox{.}
\end{equation}

Call $b=d-a$ and $c=d-a+1$. For these values of $a$, $b$, and $c$, the quadruple $(a,b,c,d)$ satisfies inequalities (\ref{Cequation.0.1}), (\ref{Cequation.0.1.5}), and (\ref{Cequation.0.2}), and there exists an $(a,b,c,d)$-pair $A$. In fact, $A$ turns out to be unique. More precisely, $A^-$ has two symmetric combs with $a+1$ teeth at vertices $0$ and $\bar{0}$, while its interior edges not involved in the combs form two centrally symmetric zigzags. The interior edges of $A^+$ simply form a zigzag. By Theorem \ref{Ctheorem.1},
\begin{equation}\label{Ctheorem1.2.eq.1}
\delta(A)\geq\frac{5}{2}(d-a)-\frac{d+2}{a}-4\mbox{.}
\end{equation}

In the remainder of the proof, it is shown that $a$ can be chosen in such a way that $\delta(A)\geq5d/2-4\sqrt{d}-4$. This lower bound on $\delta(A)$ can be derived from (\ref{Ctheorem1.2.eq.1}) if and only if $a$ satisfies the following inequality:
\begin{equation}\label{Ctheorem1.2.eq.2}
\frac{5}{2}a+\frac{d+2}{a}\leq4\sqrt{d}\mbox{.}
\end{equation}

Solving (\ref{Ctheorem1.2.eq.2}) for $a$ yields:
\begin{equation}\label{Ctheorem1.2.eq.3}
|a-\frac{4}{5}\sqrt{d}|\leq\frac{\sqrt{6d-20}}{5}\mbox{.}
\end{equation}

As $d>5$, the right-hand side of (\ref{Ctheorem1.2.eq.3}) is greater than $1/2$. Moreover, $4\sqrt{d}/5$ is bounded below by $3/2$, and above by $d/2-1$. Therefore, there exists an integer solution $a$ to inequality (\ref{Ctheorem1.2.eq.2}) that also satisfies (\ref{Ctheorem.1.2.eq.1}). \qed
\end{proof}

Observe that $5d/2-4\sqrt{d}-4$ is negative when $1\leq{d}\leq5$. Hence, because of the correspondence between the graphs of cyclohedra and the centrally symmetric triangulations of convex polygons and their flips, the following result is a direct consequence of Theorem \ref{Ctheorem.1.2}:

\begin{corollary}\label{Ccorollary.2}
For any positive integer $d$, the diameter of the $d$-dimensional cyclohedron is at least $5d/2-4\sqrt{d}-4$.
\end{corollary}

\section{Vertex deletions}
\label{Csection.2}

In order to prove Theorem \ref{Ctheorem.1.2}, a set of recursive inequalities on the distance of $(a,b,c,d)$-pairs will be established in Section \ref{Csection.3}. These inequalities will typically bound the difference between the distance of an $(a,b,c,d)$-pair and the distance of an $(a,b',c',d')$-pair below by some positive quantity, where $d'$ is less than $d$. This kind of inequalities will be obtained using the operation of deleting a vertex from a triangulation, already used in \cite{Rambau1997} for triangulations of cyclic polytopes and in \cite{Pournin2014} for triangulations of convex polygons. This section begins with a definition of this operation in the case of centrally symmetric triangulations of convex polygons. Throughout the section, $\pi$ is a convex polygon with an even number of vertices labeled in an arbitrary way. In particular, these labels are not necessarily consecutive integers.

Let $T$ be a centrally symmetric triangulation of $\pi$. Consider a boundary edge of $\pi$ and label the vertices of this edge by $p$ and $q$, in such a way that $q$ immediately follows $p$ clockwise. Any oriented pair such as $(p,q)$ will be called a clockwise oriented boundary edge of $\pi$. In \cite{Pournin2014}, deleting vertex $p$ from $T$ consists in removing $\{p,q\}$ from $T$ and replacing $p$ by $q$ within every other edge of $T$. Informally, this operation amounts to displace a vertex of $\pi$ to its clockwise immediate successor.

When $\pi$ has at least four vertices, deleting (in the sense of \cite{Pournin2014}) a single vertex from $T$ results in a triangulation of a convex polygon with one vertex less (see Proposition 2 in \cite{Pournin2014}). Observe, however, that the triangulation obtained from this operation cannot be centrally-symmetric because the number of its vertices is odd. For this reason, a different deletion operation needs be defined that affects two opposite vertices of $\pi$. More precisely:
\begin{definition}\label{Cdefinition.4}
The operation of deleting vertex $p$ from $T$ consists in removing edges $\{p,q\}$ and $\{\bar{p},\bar{q}\}$ from $T$, and replacing $p$ and $\bar{p}$ by respectively $q$ and $\bar{q}$ within all the other edges of $T$. The resulting set of edges is called $T\contract{p}$.
\end{definition}

The deletion operation defined here in the case of centrally symmetric triangulation alternatively consists in performing two consecutive deletions in the sense of \cite{Pournin2014}. Therefore, by Proposition 2 from \cite{Pournin2014}, $T\contract{p}$ is a triangulation as soon as $T$ is a triangulation with at least $6$ vertices. It also follows from Definition \ref{Cdefinition.4} that $T\contract{p}$ is centrally symmetric.

Consider two triangulations $T^-$ and $T^+$ of $\pi$, and call $P$ the pair $\{T^-,T^+\}$. The deletion of vertex $p$ can be extended to $P$ as follows:
$$
P\contract{p}=\{T^-\contract{i},T^+\contract{p}\}\mbox{,}
$$

Let $T$ be a triangulation of $\pi$. Call $r$ the vertex of $\pi$ so that $\{p,r\}$ and $\{q,r\}$ are two edges of $T$. A flip that transforms $T$ into another triangulation of $\pi$ will be called \emph{incident to edge $\{p,q\}$} when this flip removes either $\{p,r\}$ or $\{q,r\}$. In other words, such flips modify the triangle of $T$ incident to edge $\{p,q\}$. According to the next lemma, for any geodesic between triangulations $T^-$ and $T^+$, there exists a path between $T^-\contract{p}$ and $T^+\contract{p}$, shorter by the number of flips incident to $\{p,q\}$ along the geodesic.

\begin{lemma}\label{Clemma.1}
Let $P$ be a pair of centrally symmetric triangulations $T^-$ and $T^+$ of a convex polygon. Further consider a clockwise oriented boundary edge $(p,q)$ of this polygon. If there exists a geodesic between $T^-$ and $T^+$ along which at least $f$ flips are incident to edge $\{p,q\}$, then
$$
\delta(P)\geq\delta(P\contract{p})+f\mbox{.}
$$
\end{lemma}
\begin{proof}
Let $(T_i)_{0\leq{i}\leq{k}}$ be a geodesic from $T^-$ to $T^+$. Consider the sequence of triangulations $T_0\contract{p}$ to $T_k\contract{p}$. Two consecutive triangulation in this sequence are either identical or can be obtained from one another by a flip. More precisely, if the flip that transforms $T_{i-1}$ into $T_i$ is incident to $\{p,q\}$, then the deletion of vertex $p$ sends these two triangulations to the same triangulation. Indeed, in this case, the quadrilateral affected by this flip is shrunk to a triangle by the deletion. If the flip that transforms $T_{i-1}$ into $T_i$ is not incident to $\{p,q\}$, then the quadrilateral affected by a flip remains a quadrilateral after the deletion and $T_{i-1}\contract{p}$ and $T_i\contract{p}$ are still related by a flip.

Therefore, removing unnecessary triangulations from the sequence of triangulations $T_0\contract{p}$, ..., $T_k\contract{p}$, one obtains a path of length $k'$ between $T^-\contract{p}$ and $T^+\contract{p}$. As the number of triangulations that have been removed in this process is also the number $f$ of flips incident to $\{p,q\}$ along path $(T_i)_{0\leq{i}\leq{k}}$, one obtains $k'=k-f$. By definition, $k'$ is at least $\delta(P\contract{p})$. As, in addition, $k$ is equal to $\delta(P)$, then the desired result holds. \qed
\end{proof}

Under some conditions on two triangulations, some boundary edge must be incident to at least two flip along any geodesic between then. For instance, the following lemma in a consequence of Lemma \ref{Clemma.1}:

\begin{lemma}\label{Clemma.2}
Let $P$ be a pair of centrally symmetric triangulations $T^-$ and $T^+$ of a convex polygon. Further consider two clockwise oriented boundary edges $(p_0,p_1)$ and $(p_1,p_2)$ on this polygon. If the triangles of $T^-$ incident to $\{p_0,p_1\}$ and to $\{p_1,p_2\}$ do not share an edge, and if $T^+$ has an ear in $p_1$, then there exists a vertex $x\in\{p_0,p_1\}$ so that
$$
\delta(P)\geq\delta(P\contract{x})+2\mbox{.}
$$
\end{lemma}
\begin{proof}
Assume that the triangles of $T^-$ incident to $\{p_0,p_1\}$ and to $\{p_1,p_2\}$ do not share an edge, and that $T^+$ has an ear in $p_1$. Consider a geodesic $(T_i)_{0\leq{i}\leq{k}}$ from $T^-$ to $T^+$. If there are at least $2$ flips incident to $\{p_0,p_1\}$ along this geodesic then, according to Lemma \ref{Clemma.1}, 
$$
\delta(P)\geq\delta(P\contract{p_0})+2\mbox{.}
$$

Hence the desired result holds with $x=p_0$. Assume that there is at most one flip incident to $\{p_0,p_1\}$ along $(T_i)_{0\leq{i}\leq{k}}$. In this case, there must be exactly one such flip. Indeed, by hypothesis, the triangle of $T^+$ incident to $\{p_0,p_1\}$ is the ear at vertex $p_1$.
\begin{figure}
\begin{centering}
\includegraphics{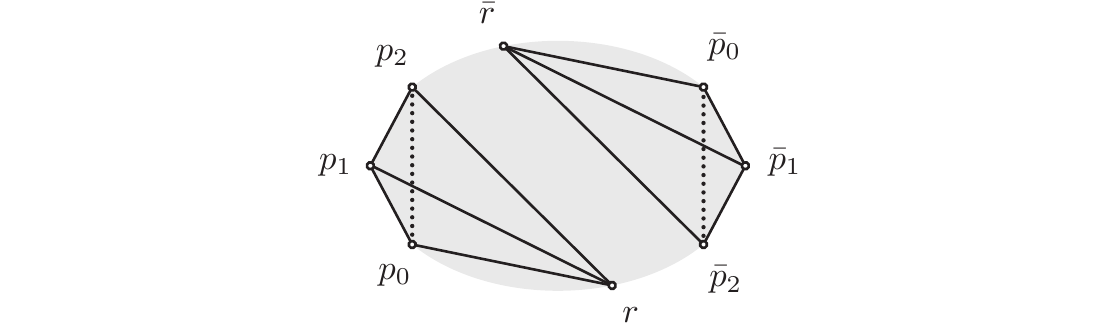}
\caption{The $j$-th flip along path $(T_i)_{0\leq{i}\leq{k}}$ used in the proof of Lemma \ref{Clemma.2}.}\label{Cfigure.Lem2}
\end{centering}
\end{figure}
Since the triangles of $T^-$ incident to $\{p_0,p_1\}$ and to $\{p_1,p_2\}$ do not share an edge, they must be distinct from the ear at vertex $p_1$. Therefore, some flip along $(T_i)_{0\leq{i}\leq{k}}$ must be incident to $\{p_0,p_1\}$.

Assume that the flip incident to $\{p_0,p_1\}$ along $(T_i)_{0\leq{i}\leq{k}}$ is the $j$-th one. This flip must then replace $\{p_1,r\}$ by $\{p_0,p_2\}$ as shown in Fig. \ref{Cfigure.Lem2}, where $r$ is the vertex of $\pi$ so that $\{p_0,r\}$ and $\{p_1,r\}$ belong to $T^-$. Observe that this flip is incident to $\{p_1,p_2\}$. Moreover, as the triangles of $T_{j-1}$ incident to $\{p_0,p_1\}$ and $\{p_1,p_2\}$ have a common edge, at least one of the first $j-1$ flips along $(T_i)_{0\leq{i}\leq{k}}$ must be incident to $\{p_1,p_2\}$. Hence, at least $2$ flips along $(T_i)_{0\leq{i}\leq{k}}$ are incident to this edge, and according to Lemma \ref{Clemma.1}, 
$$
\delta(P)\geq\delta(P\contract{p_1})+2\mbox{.}
$$

As a consequence, the desired result holds with $x=p_1$. \qed
\end{proof}

The previous lemma can be generalized to sequences of deletions:

\begin{lemma}\label{Clemma.3}
Let $P$ be a pair of centrally symmetric triangulations $T^-$ and $T^+$ of a convex polygon. Further consider $n\geq2$ clockwise oriented boundary edges $(p_0,p_1)$, ..., $(p_{n-1},p_n)$ on this polygon. If the triangles of $T^-$ incident to $\{p_0,p_1\}$, ..., $\{p_{n-1},p_n\}$ do not have common edges, and if $\{p_0,p_n\}$ is an edge of $T^+$, then a pair $Q$ of centrally symmetric triangulations, obtained from $P$ by deleting all but one of the vertices $p_0$, ..., $p_{n-1}$ satisfies
$$
\delta(P)\geq\delta(Q)+2(n-1)\mbox{.}
$$
\end{lemma}
\begin{proof}
Assume that the triangles of $T^-$ incident to $\{p_0,p_1\}$, ..., $\{p_{n-1},p_n\}$ do not have common edges, and that $T^+$ contains $\{p_0,p_n\}$. The lemma will be proven by induction on $n$. First observe that if $n=2$, then the result immediately follows from Lemma \ref{Clemma.2}.

Now assume that $n>2$. As $\{p_0,p_n\}$ is an edge of $T^+$, this triangulation induces a triangulation $U$ of the polygon with vertices $p_0$ to $p_n$. Any triangulation of a polygon with at least four vertices has at least two ears. Hence, $U$ has at least two ears and at least one of them is an ear at some vertex $p_i$ so that $0<i<n$. By assumption, the triangles of $T^-$ incident to $\{p_{i-1},p_i\}$ and to $\{p_i,p_{i+1}\}$ do not share an edge. Therefore, Lemma \ref{Clemma.2} provides some $x\in\{p_{i-1},p_i\}$ so that the following inequality holds:
\begin{equation}\label{Clemma.3.eq.1}
\delta(P)\geq\delta(P\contract{x})+2\mbox{.}
\end{equation}

Consider the vertices $p'_0$, ..., $p'_{n-1}$ obtained by removing $x$ from $p_0$, ..., $p_n$, and by relabeling the resulting sequence in such a way that the order of the indices is preserved. By construction, the triangles of $T^-\contract{x}$ incident to $\{p'_0,p'_1\}$, ..., $\{p'_{n-1},p'_n\}$ do not have common edges, and $\{p'_0,p'_n\}$ belongs to $T^+\contract{x}$. Therefore, by induction, some pair $Q$ of centrally symmetric triangulations, obtained from $P\contract{x}$ by deleting all but one of the vertices $p'_0$, ..., $p'_{n-2}$ satisfies the following inequality:
\begin{equation}\label{Clemma.3.eq.2}
\delta(P\contract{x})\geq\delta(Q)+2(n-2)\mbox{.}
\end{equation}

The result is then obtained combining (\ref{Clemma.3.eq.1}) with (\ref{Clemma.3.eq.2}). \qed
\end{proof}

\section{Proof of the main inequality}
\label{Csection.3}

Theorem \ref{Ctheorem.1} is proven in this section using the techniques introduced in the previous section. As in \cite{Pournin2014}, different sequences of deletions will be performed within $(a,b,c,d)$-pairs to obtain recursive lower bounds on their distance. In particular, the broad family of $(a,b,c,d)$-pairs needs be stable under each of these sequences of deletions. For instance:

\begin{lemma}\label{Clemma.X1}
Consider an $(a,b,c,d)$-pair $A$. Assume that at least $3$ flips are incident to $\{0,1\}$ along some geodesic between $A^-$ and $A^+$. If $a+b/2+2<d$, then there exists an $(a,b-2,b-1,d-2)$-pair $B$ so that
$$
\delta(A)\geq\delta(B)+5\mbox{.}
$$
\end{lemma}
\begin{proof}
Consider triangulation $A^-\contract{0}$, depicted on the left of Fig. \ref{Cfigure.2} and note that the deletion of vertex $0$ has merged the combs of $A^-$ at vertices $0$ and $1$.
\begin{figure}[b]
\begin{centering}
\includegraphics{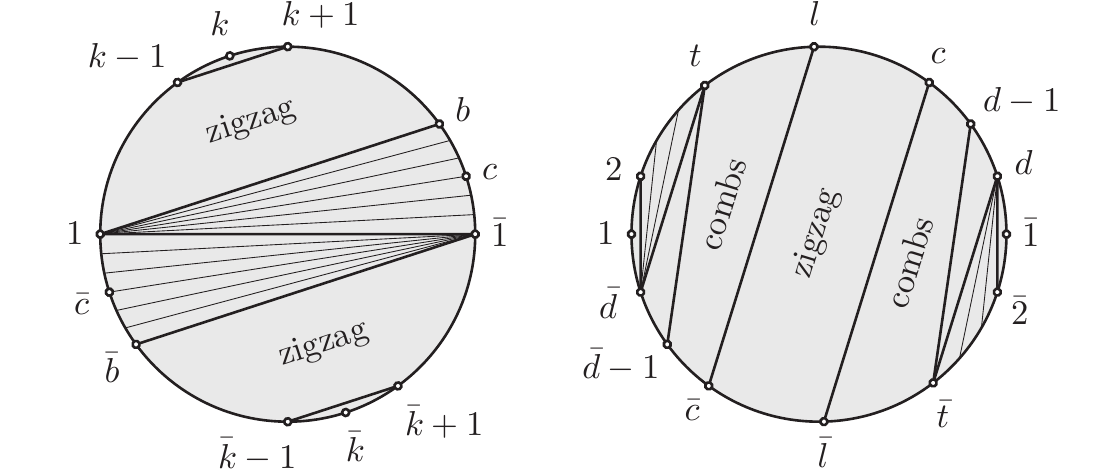}
\caption{Sketch of triangulations $A^-\contract{0}$ (left) and $A^+\contract{0}$ (right), where $k$ and $l$ are respectively equal to $\lfloor{b/2}\rfloor+1$ and $a+b-c+4$. For the sake of clarity, the placement of the vertices slightly differs in the two triangulations.}\label{Cfigure.2}
\end{centering}
\end{figure}
Further observe that, after the deletion, vertices $1$ and $\bar{1}$ immediately follow vertices $\bar{d}$ and $d$ clockwise. Now consider triangulation $A^+\contract{0}$, sketched on the right of Fig. \ref{Cfigure.2}. In this sketch, $t$ is the number of teeth of the comb at vertex $d$ in $A^+$. Observe that deleting vertex $0$ from $A^+$ only affects this comb and the symmetric comb at vertex $\bar{d}$. More precisely, $A^+\contract{0}$ has symmetric combs with $t-1$ teeth at vertices $d$ and $\bar{d}$ when $t>2$, and these combs shrink to single edges when $t=2$. According to Lemma \ref{Clemma.1}, 
\begin{equation}\label{ClemmaC1.equation.1}
\delta(A)\geq\delta(A\contract{0})+3\mbox{.}
\end{equation}

Observe that there must be at least $2$ flips incident to edge $\{d-1,d\}$ along any geodesic from $A^-\contract{0}$ to $A^+\contract{0}$. Indeed, otherwise the single such flip would replace $\{1,d\}$ by $\{d-1,\bar{t}\}$. In particular, the triangulation within which this flip is performed would contain both $\{1,d\}$ and $\{1,\bar{t}\}$. By symmetry, it would also contain edge $\{\bar{1},t\}$, which is impossible because this edge crosses $\{1,d\}$. Denote by $B$ the pair of triangulations obtained by deleting vertex $d-1$ from $A\contract{0}$, and by relabeling the vertices of the two triangulations in the resulting pair clockwise from $0$ to $2d-3$ in such a way that vertex $1$ is relabeled $0$. Since at least $2$ flips are incident to edge $\{d-1,d\}$ along any geodesic between $A^-\contract{0}$ and $A^+\contract{0}$, then according to Lemma \ref{Clemma.1},
\begin{equation}\label{ClemmaC1.equation.2}
\delta(A\contract{0})\geq\delta(B)+2\mbox{.}
\end{equation}

Finally, assume that $a+b/2+2$ is less than $d$. Under this assumption, $B$ is an $(a,b-2,b-1,d-2)$-pair. In particular, quadruple $(a,b-2,b-1,d-2)$ satisfies inequalities (\ref{Cequation.0.1}), (\ref{Cequation.0.1.5}), and (\ref{Cequation.0.2}) precisely because $a+b/2+2<d$. Moreover, the successive deletions of vertices  $0$ and $d-1$ do not change the total number of interior teeth of $A^-$ and $A^+$. Indeed, deleting vertex $0$ from $A^-$ creates a new interior tooth as this deletion merges the combs at vertices $0$ and $1$. However, the deletion of vertex $d-1$ subsequently removes an interior tooth from the merged comb. On the other hand, when $t=2$, these deletions do not remove from or add interior teeth to $A^+$, and when $t>2$, the first one removes an interior tooth from the comb at vertex $d$ in $A^+$, while the second one creates a new interior tooth as it merges the combs at vertices $d-1$ and $d$.

Therefore, combining (\ref{ClemmaC1.equation.1}) and (\ref{ClemmaC1.equation.2}) completes the proof.\qed
\end{proof}

Observe that Lemma \ref{Clemma.X1} requires the existence of a particular path between the two triangulations in an $(a,b,c,d)$-pair, which is a rather strong condition. Other conditions will be investigated in order to exhaust all possibilities. The following lemma deals with the case when $c$ is equal to $d$:

\begin{lemma}\label{Clemma.X3}
Consider an $(a,b,c,d)$-pair $A$ and call $t$ the number of interior edges of $A^-$ incident to vertex $d$. If $a+b/2+2<d$ and if $c=d$, then there exists an $(a,b-2t+2,b-2t+3,d-t)$-pair $B$ so that
$$
\delta(A)\geq\delta(B)+2t-1\mbox{.}
$$
\end{lemma}
\begin{proof}
At least two flips are incident to $\{\bar{0},d\}$ along any geodesic between $A^-$ and $A^+$. Indeed, otherwise the single such flip would replace the diagonal $\{0,\bar{0}\}$ of $A^-$ by edge $\{\bar{1},d\}$. However, since $d>1$, the latter edge cannot be a diagonal. Hence, according to Lemma \ref{Clemma.1},
\begin{equation}\label{Clemma.7.eq1}
\delta(A)\geq\delta(A\contract{d})+2\mbox{.}
\end{equation}

Assume that $c$ is equal to $d$. Triangulations $A^-\contract{d}$ and $A^+\contract{d}$ are depicted in Fig. \ref{Cfigure.4} in this case.
\begin{figure}[b]
\begin{centering}
\includegraphics{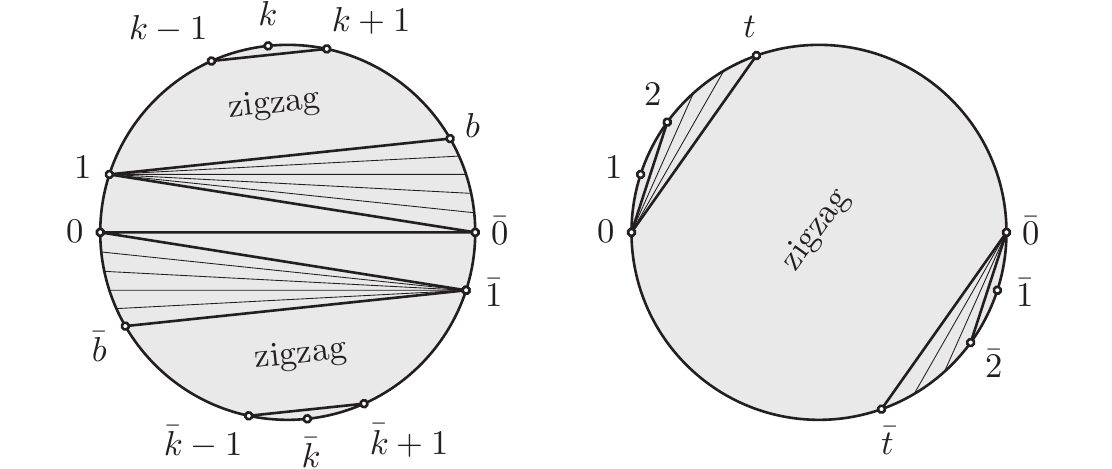}
\caption{Sketch of triangulations $A^-\contract{d}$ (left) and $A^+\contract{d}$ (right) when $c=d$, where $k$ is equal to $\lfloor{b/2}\rfloor+1$. Note that $t=a+b-c+4$ in this case. For the sake of clarity, the placement of the vertices slightly differs in the two triangulations.}\label{Cfigure.4}
\end{centering}
\end{figure}
Further assume that $a+b/2+2<d$. Consider the triangulations $B^-$ and $B^+$ obtained by deleting vertices $0$ to $t-2$ from $A^-\contract{d}$ and from $A^+\contract{d}$, and by relabeling the vertices of the resulting triangulations clockwise from $0$ to $2d-2t+1$ in such a way that vertices $t-1$ and $t$ are respectively relabeled $0$ and $1$. Denote by $B$ the pair $\{B^-,B^+\}$. One can see on Fig. \ref{Cfigure.4} that $B$ is an $(a,b-2t+2,b-2t+3,d-t)$-pair. Note, in particular, that the quadruple $(a,b-2t+2,b-2t+3,d-t)$ satisfies inequalities (\ref{Cequation.0.1}), (\ref{Cequation.0.1.5}), and (\ref{Cequation.0.2}) because $a+b/2+2<d$. Moreover, deleting vertex $d$ and vertices $0$ to $t-2$ does not change the total number of interior teeth of triangulations $A^-$ and $A^+$. Indeed, when these deletions are carried out within $A^-$, they each create a new interior teeth, except for the deletions of vertices  $d$ and $0$. When they are carried out within $A^+$, they each remove an interior teeth, except for the deletions of vertices $t-3$ and $t-2$. 

The remainder of the proof consists in showing that
\begin{equation}\label{Clemma.7.eq2}
\delta(A\contract{d})\geq\delta(B)+2t-3\mbox{.}
\end{equation}

The desired result is then obtained by combining inequalities (\ref{Clemma.7.eq1}) and (\ref{Clemma.7.eq2}). First assume that $t=2$. In this case $B^-$ and $B^+$ are obtained by deleting vertex $0$ from $A^-\contract{d}$ and from $A^+\contract{d}$ and by relabeling the vertices of the resulting triangulations. Observe that edge $\{0,1\}$ is not incident to the same triangle in $A^-\contract{d}$ and in $A^+\contract{d}$. Hence, at least one flip is incident to this edge along any geodesic between these triangulations, and according to Lemma \ref{Clemma.1}, the distance of $B^-$ and $B^+$ is less by at least one than that of $A^-\contract{d}$ and in $A^+\contract{d}$, which proves inequality (\ref{Clemma.7.eq2}) in this case.

Now assume that $t>2$. In this case, pair $A\contract{d}$ satisfies the requirements of Lemma \ref{Clemma.3} with $p_i=i$ and $n=t-1$. Therefore, the pair $C$ of the triangulations $C^-$ and $C^+$ obtained from $A^-\contract{d}$ and $A^+\contract{d}$ by deleting vertices $0$, ..., $t-2$ except, say, vertex $r$, satisfies
\begin{equation}\label{Clemma.7.eq3}
\delta(A\contract{d})\geq\delta(C)+2(t-2)\mbox{.}
\end{equation}

Observe that $\{r,t-1\}$ is a boundary edge of $C^-$ and $C^+$ that is not incident to the same triangle in these two triangulations. More precisely, the triangle incident to $\{r,t-1\}$ in $C^-$ has a vertex greater than $t$, while its counterpart in $C^+$ has vertices $r$, $t-1$, and $t$. Hence, at least one flip is incident to $\{r,t-1\}$ along any geodesic between $C^-$ and $C^+$. By construction, $B^-$ and $B^+$ are obtained by deleting vertex $r$ from $C^-$ and from $C^+$, and by relabeling the vertices of the resulting triangulations. As a consequence, by Lemma \ref{Clemma.1},
\begin{equation}\label{Clemma.7.eq4}
\delta(C)\geq\delta(B)+1\mbox{.}
\end{equation}

Combining (\ref{Clemma.7.eq3}) with (\ref{Clemma.7.eq4}) therefore proves inequality (\ref{Clemma.7.eq2}). \qed
\end{proof}

The requirements of the next theorem (Theorem \ref{Ctheorem.2} below) are complementary to those of the last two lemma. In the proof of this theorem, the following lemma will instrumental. In particular, it will be invoked twice.

\begin{lemma}\label{Csection.3.lemma.1}
Consider an $(a,b,c,d)$-pair $A$. Let $T$ be some triangulation along a geodesic between $A^-$ and $A^+$. If $T$ has an interior edge whose two vertices belong to $\{c, ..., d, \bar{0}\}$, then there exists a vertex $x$ so that $c\leq{x}\leq{d}$ and
$$
\delta(A)\geq\delta(A\contract{x})+3\mbox{.}
$$
\end{lemma}
\begin{proof}
Consider a geodesic $(T_i)_{0\leq{i}\leq{k}}$ from $A^-$ to $A^+$ and consider an integer $j$ so that $0\leq{j}\leq{k}$. Assume that triangulation $T_j$ has an interior edge whose two vertices belong to $\{c, ..., d, \bar{0}\}$. Denote these vertices by $q$ and $r$ with the convention that $q$ is less than $r$.

In this case, triangulation $T_j$ necessarily has an ear at some vertex $p$ so that $q<p<r$. Indeed, since $T_j$ contains $\{q,r\}$, it induces a triangulation $U$ of the polygon whose vertices are vertices $q$ to $r$. Any triangulation of a polygon with at least four vertices has at least two ears. Hence, if $U$ has at least four vertices, then one of its ears is an ear at some vertex $p$ so that $q<p<r$. If $U$ has exactly three vertices, then it is made up of a single triangle, and this triangle is an ear at vertex $p=q+1$. Note that $U$ cannot have less than $3$ vertices because $\{q,r\}$ is an interior edge of $T_j$. Therefore, $U$ always has an ear at some vertex $p$ so that $q<p<r$ and this ear is also an ear of $T_j$.

Now observe that triangulations $T_j$ and $A^+$ satisfy the conditions of Lemma \ref{Clemma.2}. Note, in particular, that the triangles of $A^+$ incident to $\{p-1,p\}$ and to $\{p,p+1\}$ do not share an edge because, by definition, $A^+$ has a comb at vertex $p$. Therefore, there exists $x\in\{p-1,p\}$ so that
\begin{equation}\label{Csection3.lemma.1.eq.1}
k-j\geq\delta(\{T_j\contract{x},A^+\contract{x}\})+2\mbox{.}
\end{equation}

Further note that $\{x,x+1\}$ is not incident to the same triangle in $A^-$ and in $T_j$ because $A^-$ does not have an ear at vertex $p$. Hence, at least one of the first $j$ flips along $(T_i)_{0\leq{i}\leq{k}}$ is incident to $\{x,x+1\}$. By Lemma \ref{Clemma.1},
\begin{equation}\label{Csection3.lemma.1.eq.2}
j\geq\delta(\{A^-\contract{x},T_j\contract{x}\})+1\mbox{.}
\end{equation}

Since $k$ is precisely the distance of pair $A$, combining (\ref{Csection3.lemma.1.eq.1}) with (\ref{Csection3.lemma.1.eq.2}) and using the triangle inequality yields
$$
\delta(A)\geq\delta(A\contract{x})+3\mbox{.}
$$

Finally, as $q<p<r$ and as $q$ and $r$ belong to $\{c, ..., d, \bar{0}\}$, vertex $p$ must satisfy $c<p\leq{d}$. Since $x\in\{p-1,p\}$, this proves that $c\leq{x}\leq{d}$. \qed
\end{proof}

The following theorem, whose conditions are complementary to those of Lemmas \ref{Clemma.X1} and \ref{Clemma.X3} can now be proven. It provides the first deletion for a sequence of deletions under which the family of $(a,b,c,d)$-pairs is stable:

\begin{theorem}\label{Ctheorem.2}
Consider an $(a,b,c,d)$-pair $A$ so that $c<d$. If at most $2$ flips are incident to edge $\{0,1\}$ along a geodesic between $A^-$ and $A^+$, then there exists a vertex $x$ of these triangulations so that $c\leq{x}\leq{d}$ and
$$
\delta(A)\geq\delta(A\contract{x})+3\mbox{.}
$$
\end{theorem}
\begin{proof}
Consider some geodesic $(T_i)_{0\leq{i}\leq{k}}$ from $A^-$ to $A^+$. Assume that $c<d$ and that that at most $2$ flips are incident to edge $\{0,1\}$ along $(T_i)_{0\leq{i}\leq{k}}$. If $3$ or more flips are incident to $\{\bar{0},d\}$ along this geodesic, then Lemma \ref{Clemma.1} immediately provides the desired result with $x=d$.
\begin{figure}
\begin{centering}
\includegraphics{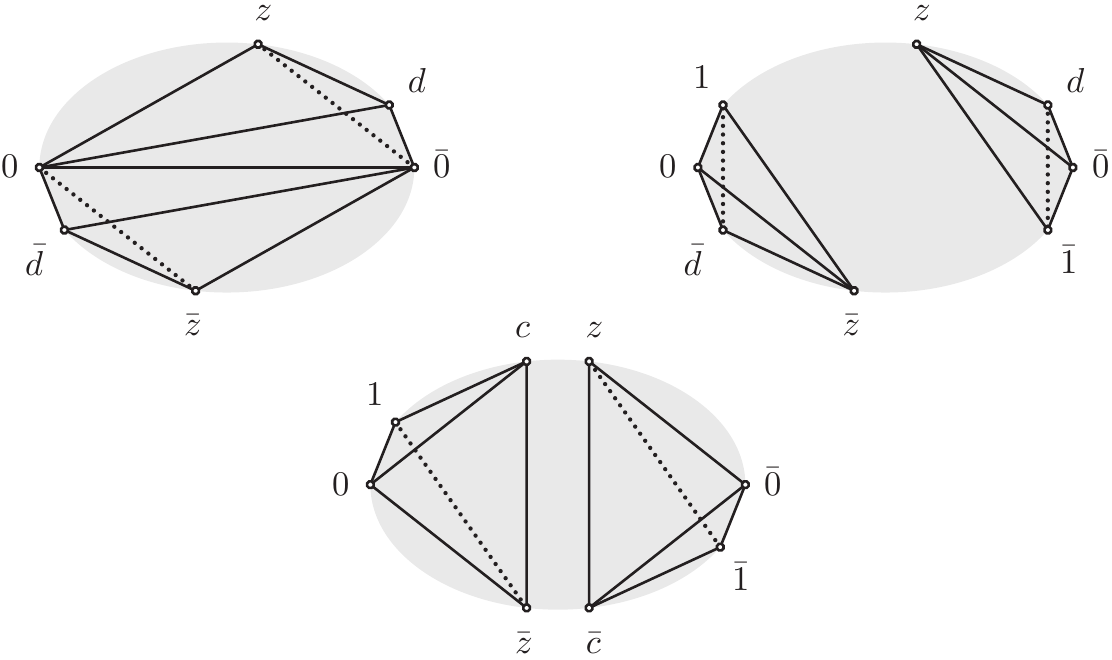}
\caption{The $j$-th (top left), $j'$-th (top right), and $j''$-th (bottom) flips along path $(T_i)_{0\leq{i}\leq{k}}$ used in the proof of Theorem \ref{Ctheorem.2}, in case the first of these flips introduces edge $\{\bar{0},z\}$ with $z<d$. The edges introduced by each flip are dotted.}\label{Cfigure.thm2.1}
\end{centering}
\end{figure}
Hence, it is further assumed that at most $2$ flips are incident to edge $\{\bar{0},d\}$ along $(T_i)_{0\leq{i}\leq{k}}$.

Observe that $\{\bar{0},d\}$ is not incident to the same triangle in $A^-$ and in $A^+$. Hence, at least one flip must be incident to $\{\bar{0},d\}$ along $(T_i)_{0\leq{i}\leq{k}}$. In fact, there must be exactly two such flips. Indeed, the unique such flip would otherwise replace the diagonal $\{0,\bar{0}\}$ by edge $\{\bar{1},d\}$. However, since $d$ is greater than $1$, the latter edge cannot be a diagonal and such a flip is therefore impossible. Assume that the two flips incident to edge $\{\bar{0},d\}$ along $(T_i)_{0\leq{i}\leq{k}}$ are the $j$-th one and the $j'$-th one, with $j<j'$. Observe that the $j$-th flip along $(T_i)_{0\leq{i}\leq{k}}$ either introduces edge $\{\bar{0},z\}$ where $0<z<d$ or edge $\{d,\bar{z}\}$ where $0<z\leq{d}$. The two cases will be reviewed separately.

First assume that the $j$-th flip along $(T_i)_{0\leq{i}\leq{k}}$ introduces edge $\{\bar{0},z\}$ where $0<z<d$. This flip must then be as shown top left on Fig. \ref{Cfigure.thm2.1}. In this case, the $j'$-th flip along $(T_i)_{0\leq{i}\leq{k}}$ necessarily replaces $\{\bar{0},z\}$ by $\{\bar{1},d\}$ as shown top right on the same figure. Observe that the latter flip is incident to $\{0,1\}$. There must be at least one other flip incident to this edge, taking place earlier along path $(T_i)_{0\leq{i}\leq{k}}$ because edge $\{0,1\}$ is incident to distinct triangles in $A^-$ and in $T_{j'-1}$. Say the first flip incident to $\{0,1\}$ along $(T_i)_{0\leq{i}\leq{k}}$ is the $j''$-th one. Since at most $2$ flips along path $(T_i)_{0\leq{i}\leq{k}}$ are incident to edge $\{0,1\}$, then the $j'$-th and the $j''$-th flip are the only two such flips. In particular, the latter flip must replace the triangle of $A^-$ incident to $\{0,1\}$ by the triangle of $T_{j'-1}$ incident to the same edge. Therefore this flip removes $\{0,c\}$ and replaces it by $\{1,\bar{z}\}$ as shown in the bottom of Fig. \ref{Cfigure.thm2.1}. As one can see in the figure, $z$ cannot be less than $c$. Indeed, triangulation $T_{j''-1}$ would otherwise contain crossing edges, as for instance $\{0,c\}$ and $\{\bar{0},z\}$.

This shows that $\{\bar{0},z\}$ is an interior edge of $T_{j''}$ whose two vertices belong to $\{c, ..., d, \bar{0}\}$. Therefore, the desired result follows from Lemma \ref{Csection.3.lemma.1}.

Now assume that the $j$-th flip along path $(T_i)_{0\leq{i}\leq{k}}$ introduces edge $\{d,\bar{z}\}$ where $0<z\leq{d}$. Observe that this flip removes the diagonal $\{0,\bar{0}\}$ from $A^-$. Therefore, $\{d,\bar{z}\}$ must also be a diagonal, which proves that $z=d$.
\begin{figure}
\begin{centering}
\includegraphics{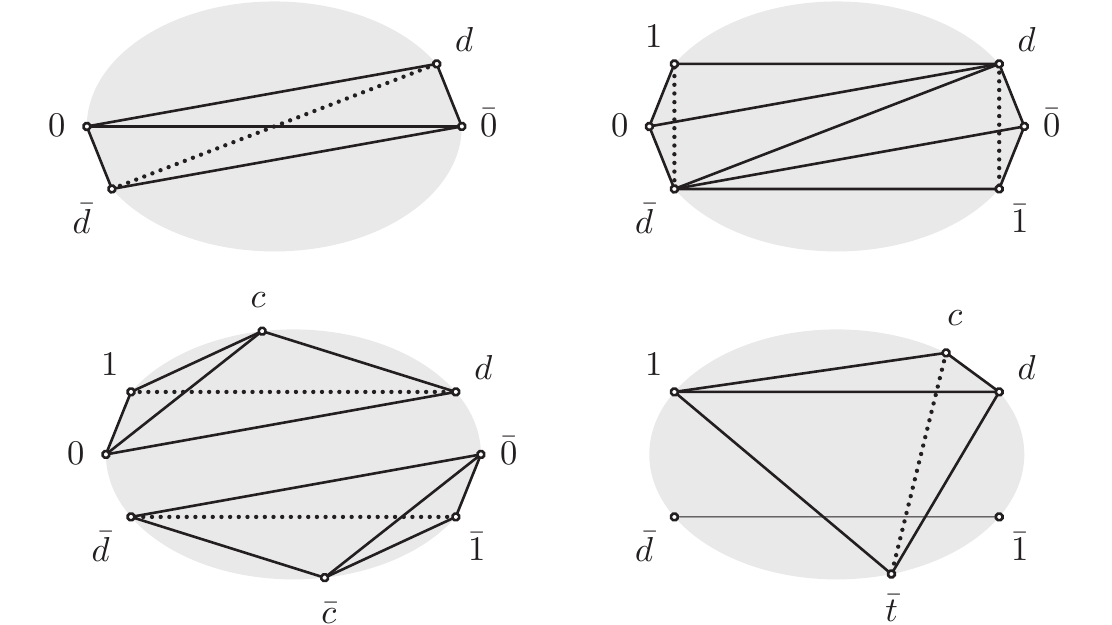}
\caption{The $j$-th (top left), $j'$-th (top right), and $j''$-th (bottom left) flips along path $(T_i)_{0\leq{i}\leq{k}}$ used in the proof of Theorem \ref{Ctheorem.2} when the first of these flips introduces edge $\{d,\bar{z}\}$ where $0<z\leq{d}$. The flip shown bottom right cannot be performed within a centrally symmetric triangulation, because some edges of this triangulation would otherwise be crossing. The edges introduced by each flip are dotted.}\label{Cfigure.thm2.2}
\end{centering}
\end{figure}
In this case, the $j$-th flip along path $(T_i)_{0\leq{i}\leq{k}}$ is necessarily the one shown top left on Fig. \ref{Cfigure.thm2.2}, and the $j'$-th flip along $(T_i)_{0\leq{i}\leq{k}}$ must replace $\{\bar{0},\bar{d}\}$ by $\{\bar{1},d\}$ as shown top right on the same figure. Observe that the latter flip is incident to edge $\{0,1\}$. There must be at least one other flip incident to this edge, taking place earlier along path $(T_i)_{0\leq{i}\leq{k}}$. Indeed, edge $\{0,1\}$ is incident to distinct triangles in $A^-$ and in $T_{j'-1}$ because $c$ is less than $d$. Say the first flip incident to $\{0,1\}$ along $(T_i)_{0\leq{i}\leq{k}}$ is the $j''$-th one. Since at most $2$ flips along path $(T_i)_{0\leq{i}\leq{k}}$ are incident to edge $\{0,1\}$, then the $j'$-th and the $j''$-th flip are the only two such flips. In particular, the latter flip must replace edge $\{0,c\}$ by edge $\{1,d\}$ as shown bottom left on Fig. \ref{Cfigure.thm2.2}.

Observe that triangulation $T_{j''}$ contains edge $\{c,d\}$. If $c$ is less than $d-1$, then this edge is an interior edge of $T_{j''}$ and the result follows from Lemma \ref{Csection.3.lemma.1}. If $c$ is equal to $d-1$, then $\{c,d\}$ is a boundary edge of $T_{j''}$ and the $j''$-th flip along $(T_i)_{0\leq{i}\leq{k}}$ is incident to it. Observe that $\{c,d\}$ is not incident to the same triangle in $T_{j''}$ and in $A^+$. Hence, at least one of the last $k-j''$ flips along $(T_i)_{0\leq{i}\leq{k}}$ is incident to edge $\{c,d\}$. In fact, there must be at least two such flips. Otherwise, the unique such flip must replace $\{1,d\}$ by $\{c,\bar{t}\}$ as shown bottom right on Fig. \ref{Cfigure.thm2.2}, where $t$ is the vertex so that edges $\{c,\bar{t}\}$ and $\{d,\bar{t}\}$ belong to triangulation $A^+$. In particular, $1<t<d$. This flip cannot occur within a centrally symmetric triangulation. Indeed, by symmetry, this triangulation would then also contain edge $\{\bar{1},\bar{d}\}$, sketched as a thin line in Fig. \ref{Cfigure.thm2.2}. However, $\{\bar{1},\bar{d}\}$ crosses at least two edges of the triangulation as, for instance edges $\{1,\bar{t}\}$ and $\{d,\bar{t}\}$.

This shows that at least two of the last $k-j''$ flips along $(T_i)_{0\leq{i}\leq{k}}$ are incident to edge $\{c,d\}$. Since the $j''$-th flip along this path is also incident to $\{c,d\}$, then Lemma \ref{Clemma.1} provides the desired result with $x=c$. \qed
\end{proof}

Consider an $(a,b,c,d)$-pair $A$ satisfying the requirements of Theorem \ref{Ctheorem.2}. When the vertex $x$ provided by this theorem satisfies $c\leq{x}<d$, then its deletion from $A^-$ and from $A^+$ results in an $(a,b,c,d-1)$-pair up to the vertex labels of the resulting triangulations. In particular:

\begin{lemma}\label{Clemma.X4}
Consider an $(a,b,c,d)$-pair $A$. If some vertex $x$ of triangulations $A^-$ and $A^+$ satisfies $c\leq{x}<d$, then there exists an $(a,b,c,d-1)$-pair $B$ whose distance is equal to that of pair $A\contract{x}$.
\end{lemma}
\begin{proof}
Let $x$ be a vertex of triangulations $A^-$ and $A^+$ so that $c\leq{x}<d$. Consider the triangulations $B^-$ and $B^+$ obtained by relabeling the vertices of $A^-\contract{x}$ and $A^+\contract{x}$ clockwise from $0$ to $2d-1$ in such a way that vertex $0$ keeps its label. Call $B$ the pair $\{B^-,B^+\}$.

It turns out that $B$ is an $(a,b,c,d-1)$-pair. In particular, the quadruple $(a,b,c,d-1)$ satisfies (\ref{Cequation.0.1}), (\ref{Cequation.0.1.5}), and (\ref{Cequation.0.2}) because $c$ is less than $d$. Moreover, the total number of interior teeth of triangulations $A^-$ and $A^+$ is not changed by the deletion. Indeed, this deletion removes an interior tooth of the comb at vertex $0$ in $A^-$, but it also creates a new interior tooth as it merges the combs at vertices $x$ and $x+1$ within triangulation $A^+$. \qed
\end{proof}

When the vertex $x$ provided by Theorem \ref{Ctheorem.2} is equal to $d$, then its deletion only initiates a sequence of deletions that leaves the family of $(a,b,c,d)$-pairs stable. The rest of the sequence will be obtained from Lemma \ref{Clemma.3}:

\begin{lemma}\label{Clemma.X2}
Consider an $(a,b,c,d)$-pair $A$. Call $t$ the number of interior edges of $A^+$ incident to vertex $d$. If $a+b/2+2<d$ and if $c<d$, then there exists an integer $c'$ satisfying $b-2t+2<c'\leq{c-t+1}$ and an $(a,b-2t+2,c',d-t)$-pair $B$ so that the following inequality holds:
$$
\delta(A\contract{d})\geq\delta(B)+2(t-1)\mbox{.}
$$
\end{lemma}
\begin{proof}
Assume that $a+b/2+2<d$ and that $c<d$. Consider triangulations $A^-\contract{d}$ and $A^+\contract{d}$, depicted in Fig. \ref{Cfigure.3}.
\begin{figure}
\begin{centering}
\includegraphics{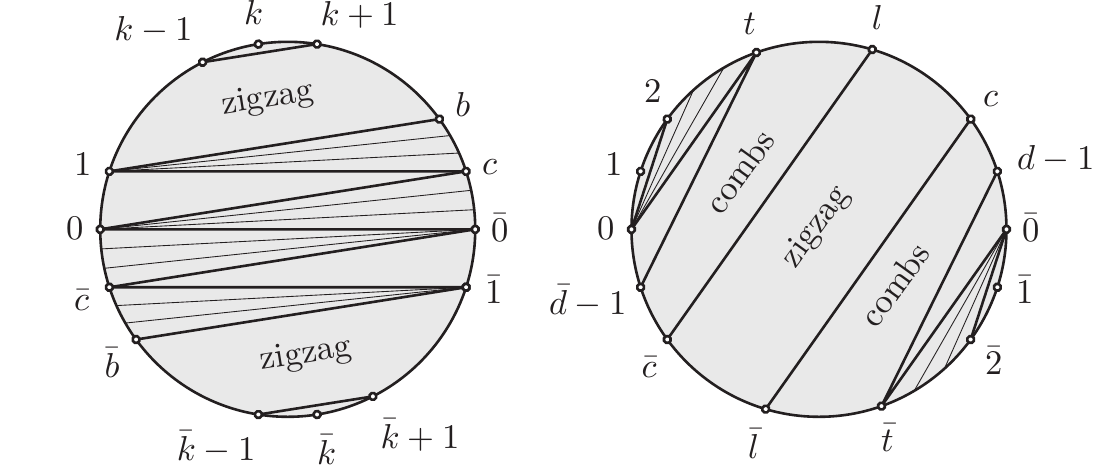}
\caption{Sketch of triangulations $A^-\contract{d}$ (left) and $A^+\contract{d}$ (right) when $c<d$, where $k$ and $l$ are respectively equal to $\lfloor{b/2}\rfloor+1$ and $a+b-c+4$. For the sake of clarity, the placement of the vertices slightly differs in the two triangulations.}\label{Cfigure.3}
\end{centering}
\end{figure}
Observe that pair $A\contract{d}$ satisfies the requirements of Lemma \ref{Clemma.3} with $p_i=i$ and $n=t$. This lemma therefore provides two triangulations $B^-$ and $B^+$ obtained from $A^-\contract{d}$ and $A^+\contract{d}$ by deleting all but one of the vertices $0$, ..., $t-1$, so that:
$$
\delta(A\contract{d})\geq\delta(B)+2(t-1)\mbox{,}
$$
where $B$ denotes the pair $\{B^-,B^+\}$. Let $r$ be the vertex among vertices $0$ to $t-1$ that is not deleted in the process. Note that, when the sequence of deletions is carried out within $A^-\contract{d}$, each of these deletions merges the comb at the deleted vertex with the comb at the next vertex clockwise. This results into a comb at vertex $r$ and a comb at vertex $t$ in $B^-$. Observe that the total number of interior teeth of $A^-$ and $A^+$ is not modified by the sequence of deletions carried out to build $B^-$ and $B^+$. Indeed, the deletion of vertex $d$ from $A^-$ removes one interior tooth from the comb at vertex $0$, and each of the other deletions creates a new interior tooth within this triangulation as these deletions each merge two combs. Therefore, recalling that these triangulations are centrally symmetric, the number of interior teeth of $B^-$ is greater by $2(t-2)$ than that of $A^-$. On the other hand, each of the deletions carried out within triangulation $A^+$ removes one tooth from the comb originally attached at vertex $d$. As these teeth are interior teeth of $A^+$ except for two of them,  $B^+$ has $2(t-2)$ interior teeth less than $A^+$.

Relabel the vertices of $B^-$ and $B^+$ clockwise from $0$ to $2d-2t+1$ in such a way that vertices $r$ and $t$ are respectively relabeled $0$ and $1$. After doing so, $B$ is an $(a,b-2t+2,c',d-t)$-pair, where $c'$ is equal to $c-t+1$ if $r=0$ and to $b-t-r+2$ otherwise. In particular, the desired bounds on $c'$ hold. Also, note that the quadruple $(a,b-2t+2,c',d-t)$ satisfies inequalities (\ref{Cequation.0.1}), (\ref{Cequation.0.1.5}), and (\ref{Cequation.0.2}) because $a+b/2+2$ and $c$ are both less than $d$. \qed
\end{proof}

Using Theorem \ref{Ctheorem.2}, and Lemmas \ref{Clemma.X1}, \ref{Clemma.X3}, \ref{Clemma.X4}, and \ref{Clemma.X2}, one can now prove Theorem \ref{Ctheorem.1}. Recall that, according to this theorem, any $(a,b,c,d)$-pair $A$ satisfies
$$
\delta(A)\geq3d-\left(\frac{b}{2}+\frac{2c-b}{a}+3a+5\right)\mbox{.}
$$
\begin{proof}[Theorem \ref{Ctheorem.1}]
Consider an $(a,b,c,d)$-pair $A$. The theorem will be proven by induction on $d-\lfloor{b/2}\rfloor$. Note that, according to inequality (\ref{Cequation.0.2}), this quantity is greater than $a+1$. 

First assume that $d-\lfloor{b/2}\rfloor$ is equal to $a+2$. Observe that $A^-$ and $A^+$ have no edge in common. Therefore, any geodesic between these triangulations has length at least $d$. Indeed, each flip along such a geodesic either removes a pair of centrally symmetric interior edges, or a diagonal. As $A^-$ has exactly $d-1$ pairs of centrally symmetric interior edges, and one diagonal, at least $d$ flips must be performed in order to remove all the edges of $A^-$, and therefore, to transform this triangulation into $A^+$:
\begin{equation}\label{Cequation.proofTh1.1}
\delta(A)\geq{d}\mbox{.}
\end{equation}

As $d=a+\lfloor{b/2}\rfloor+2$, one immediately obtains:
\begin{equation}\label{Cequation.proofTh1.1.5}
d\leq{a+\frac{b}{2}+2}\mbox{.}
\end{equation}

According to (\ref{Cequation.0.1}), $b\leq{d-1}$. Therefore, (\ref{Cequation.proofTh1.1.5}) yields:
\begin{equation}\label{Cequation.proofTh1.1.7}
d\leq{2a+3}\mbox{.}
\end{equation}

As the ratio of $2c-b$ and $a$ is positive, summing (\ref{Cequation.proofTh1.1.5}) with (\ref{Cequation.proofTh1.1.7}), and adding this ratio to the right-hand side results in the following inequality:
\begin{equation}\label{Cequation.proofTh1.2}
2d\leq\frac{b}{2}+\frac{2c-b}{a}+3a+5\mbox{.}
\end{equation}

One then obtains the desired inequality by combining (\ref{Cequation.proofTh1.1}) with (\ref{Cequation.proofTh1.2}). Now assume that $d-\lfloor{b/2}\rfloor$ is greater than $a+2$. Further assume that, for any $(a,b',c',d')$-pair $B$ so that $d'-\lfloor{b'/2}\rfloor$ is less than $d-\lfloor{b/2}\rfloor$ by $1$,
\begin{equation}\label{Cequation.proofTh1.4}
\delta(B)\geq3d'-\left(\frac{b'}{2}+\frac{2c'-b'}{a}+3a+5\right)\mbox{.}
\end{equation}

Recall that the conditions of Lemma \ref{Clemma.X1}, of Lemma \ref{Clemma.X3}, and of Theorem \ref{Ctheorem.2} are complementary. In the remainder of the proof, these conditions are reviewed one after the other, and the result is proven in each case.

First assume that at least $3$ flips are incident to $\{0,1\}$ along some geodesic between $A^-$ and $A^+$. In this case, the requirements of Lemma \ref{Clemma.X1} are satisfied (in particular, $a+b/2+2<d$ precisely because $d-\lfloor{b/2}\rfloor$ is greater than $a+2$), and this lemma provides an $(a,b',c',d')$-pair $B$ so that
\begin{equation}\label{Cequation.proofTh1.3}
\delta(A)\geq\delta(B)+5\mbox{,}
\end{equation}
where $b'$, $c'$, and $d'$ are respectively equal to $b-2$, $b-1$, and $d-2$. Note that $d'-\lfloor{b'/2}\rfloor$ is less than $d-\lfloor{b/2}\rfloor$ by $1$. Hence, by induction, inequality (\ref{Cequation.proofTh1.4}) holds for pair $B$. Combining this inequality with (\ref{Cequation.proofTh1.3}), and replacing $b'$, $c'$, and $d'$ by respectively $b-2$, $b-1$, and $d-2$, one obtains
$$
\delta(A)\geq3d-\left(\frac{b}{2}+\frac{b}{a}+3a+5\right)\mbox{.}
$$

Since $b\leq2c-b$, this proves that the result holds when at least $3$ flips are incident to $\{0,1\}$ along some geodesic between $A^-$ and $A^+$.

Now assume that $c=d$. Again, $a+b/2+2$ is less than $d$ because $d-\lfloor{b/2}\rfloor$ is greater than $a+2$. Therefore, by Lemma \ref{Clemma.X3}, there exists an $(a,b',c',d')$-pair $B$ that satisfies the following inequality:
\begin{equation}\label{Cequation.proofTh1.7}
\delta(A)\geq\delta(B)+2t-1\mbox{,}
\end{equation}
where $t$ is the number of interior edges of $A^+$ incident to vertex $d$. Moreover, $b'=b-2t+2$, $c'=b-2t+3$, and $d'=d-t$. Note that $d'-\lfloor{b'/2}\rfloor$ is less than $d-\lfloor{b/2}\rfloor$ by $1$. Hence, by induction, inequality (\ref{Cequation.proofTh1.4}) also holds in this case. Combining this inequality with (\ref{Cequation.proofTh1.7}) and replacing $b'$, $c'$ and $d'$ by respectively $b-2t+2$, $b-2t+2$ and $d-t$ yields:
\begin{equation}\label{Cequation.proofTh1.8}
\delta(A)\geq3d-\left(\frac{b}{2}+\frac{b-2t+4}{a}+3a+7\right)\mbox{.}
\end{equation}

Observe that, since $c=d$, the only combs in $A^+$ are the combs at vertices $d$ and $\bar{d}$. In particular, the number of interior teeth of $A^+$ is exactly $2t-4$. Moreover, $A^-$ has exactly $2c-2b-2$ interior teeth. As $a$ is obtained by adding $1$ to half the total number of interior teeth in $A^-$ and $A^+$, it follows that $t$ is equal to $a+b-c+2$. Replacing $t$ within (\ref{Cequation.proofTh1.8}) by the latter expression results in the desired inequality, and the result holds in this case.

Finally, assume that $c<d$ and that at most $2$ flips are incident to $\{0,1\}$ along some geodesic between $A^-$ and $A^+$. By Theorem \ref{Ctheorem.2}, there exists a vertex $x$ of $A^-$ and $A^+$ so that $c\leq{x}\leq{d}$ and
\begin{equation}\label{Cequation.proofTh1.48}
\delta(A)\geq\delta(A\contract{x})+3\mbox{.}
\end{equation}

If $c\leq{x}<d$, then Lemma \ref{Clemma.X4} provides an $(a,b',c',d')$-pair $B$ with the same distance as $A\contract{x}$, where  $b'=b$, $c'=c$, and $d'=d-1$. Note, in particular, that $d'-\lfloor{b'/2}\rfloor$ is less than $d-\lfloor{b/2}\rfloor$ by $1$. Hence, by induction, (\ref{Cequation.proofTh1.4}) holds for pair $B$. Replacing $A\contract{x}$ by $B$ in (\ref{Cequation.proofTh1.48}) and combining the resulting inequality with (\ref{Cequation.proofTh1.4}) provides the result. Now, if $x=d$, then the conditions of Lemma \ref{Clemma.X2} are satisfied. In particular, $a+b/2+2<d$ precisely because $d-\lfloor{b/2}\rfloor$ is greater than $a+2$. Hence, there exists an $(a,b',c',d')$-pair $B$ so that
\begin{equation}\label{Cequation.proofTh1.6}
\delta(A\contract{x})\geq\delta(B)+2(t-1)\mbox{,}
\end{equation}
where $t$ is the number of interior edges of $A^+$ incident to vertex $d$. Moreover, $b'=b-2t+2$, $d'=d-t$, and $c'$ satisfies
$$
b-2t+2<c'\leq{c-t+1}\mbox{.}
$$

Note that $d'-\lfloor{b'/2}\rfloor$ is less than $d-\lfloor{b/2}\rfloor$ by $1$. Therefore, by induction, inequality (\ref{Cequation.proofTh1.4}) also holds in this case. Combining (\ref{Cequation.proofTh1.4}), (\ref{Cequation.proofTh1.48}), and (\ref{Cequation.proofTh1.6}), and replacing $b'$, and $d'$ by respectively $b-2t+2$ and $d-t$ yields:
$$
\delta(A)\geq3d-\left(\frac{b}{2}+\frac{2c'-b+2t-2}{a}+3a+5\right)\mbox{.}
$$

Therefore, the result holds because $c'$ is not greater than $c-t+1$. \qed
\end{proof}

\section{Discussion}
\label{Csection.4}

It has been shown in this article that the diameter $\Delta$ of the $d$-dimensional cyclohedron is at least $5d/2-4\sqrt{d}-4$ and at most $\lceil{5d/2}\rceil-2$. In particular, this diameter grows like $5d/2$ when $d$ is large:
$$
\lim_{d\rightarrow\infty}\frac{\Delta}{d}=\frac{5}{2}\mbox{.}
$$
 
The values of $\Delta$ when $d$ is small can be obtained computationally. These values deceptively suggest a $7/3$ coefficient instead of the above $5/2$:
$$
\begin{array}{c|>{\centering\arraybackslash$} p{0.6cm} <{$}>{\centering\arraybackslash$} p{0.6cm} <{$}>{\centering\arraybackslash$} p{0.6cm} <{$}>{\centering\arraybackslash$} p{0.6cm} <{$}>{\centering\arraybackslash$} p{0.6cm} <{$}>{\centering\arraybackslash$} p{0.6cm} <{$}>{\centering\arraybackslash$} p{0.6cm} <{$}>{\centering\arraybackslash$} p{0.6cm} <{$}>{\centering\arraybackslash$} p{0.6cm} <{$}>{\centering\arraybackslash$} p{0.6cm} <{$}>{\centering\arraybackslash$} p{0.6cm} <{$}>{\centering\arraybackslash$} p{0.6cm} <{$}>{\centering\arraybackslash$} p{0.6cm} <{$}}
d & 1 & 2 & 3 & 4 & 5 & 6 & 7 & 8 & 9 & 10 & 11 & 12 & 13\\
\hline
\Delta & 1 & 3 & 5 & 7 & 9 & 11 & \bm{14} & 16 & 18 & \bm{21} & 23 & 25 & \bm{28}^\star\\
\end{array}
$$

In this table, the starred value is only a lower bound on $\Delta$ that is not necessarily sharp. An exact computation was not possible in this case due to prohibitive computation time. Each value of $\Delta$ is greater by $2$ than the preceding value, except for the ones shown in bold characters, corresponding to $d=7$ and $d=10$, that are greater by $3$ than the preceding value. Note that the lower bound on $\Delta$ when $d$ is equal to $13$, also shown in bold, is greater by $3$ than the preceding value in the table as well.

The last of the three infinite subfamilies of generalized associahedra whose diameter is not known exactly, but only asymptotically, is that of cyclohedra. However, the asymptotic behavior of this diameter is very different from its behavior at low dimensions, which suggests that evaluating the exact diameter of cyclohedra for all dimensions may turn out to be difficult.

\end{document}